\documentclass[11pt,a4paper]{article}
\usepackage{mathrsfs}
\usepackage{indentfirst}
\usepackage{amsmath,amssymb,amsfonts,amsthm,graphics}
\usepackage{makeidx}
\usepackage{color}

\newtheorem{theorem}{Theorem}

\newtheorem{lemma}{Lemma}

\newtheorem{observation}{Observation}

\begin{document}
\title{\Large\bf The minimal size of a graph with given generalized
$3$-edge-connectivity\footnote{Supported by NSFC No.11071130}}
\author{\small Xueliang Li, Yaping Mao
\\
\small Center for Combinatorics and LPMC-TJKLC
\\
\small Nankai University, Tianjin 300071, China
\\
\small lxl@nankai.edu.cn; maoyaping@ymail.com.}
\date{}
\maketitle
\begin{abstract}
For $S\subseteq V(G)$ and $|S|\geq 2$, $\lambda(S)$ is the maximum
number of edge-disjoint trees connecting $S$ in $G$. For an integer
$k$ with $2\leq k\leq n$, the \emph{generalized
$k$-edge-connectivity} $\lambda_k(G)$ of $G$ is then defined as
$\lambda_k(G)= min\{\lambda(S) : S\subseteq V(G) \ and \ |S|=k\}$.
It is also clear that when $|S|=2$, $\lambda_2(G)$ is nothing new
but the standard edge-connectivity $\lambda(G)$ of $G$. In this
paper, graphs of order $n$ such that $\lambda_3(G)=n-3$ is
characterized. Furthermore, we determine the minimal number of edges
of a graph of order $n$ with $\lambda_3=1,n-3,n-2$ and give a sharp
lower bound for $2\leq \lambda_3\leq n-4$.

\end{abstract}

\section{Introduction}

All graphs considered in this paper are undirected, finite and
simple. We refer to the book \cite{bondy} for graph theoretical
notation and terminology not described here. For a graph $G$, let
$V(G)$ and $E(G)$ denote the set of vertices and the set of edges of
$G$, respectively. As usual, the \emph{union} of two graphs $G$ and
$H$ is the graph, denoted by $G\cup H$, with vertex set $V(G)\cup
V(H)$ and edge set $E(G)\cup E(H)$. Let $mH$ be the disjoint union
of $m$ copies of a graph $H$. We denote by $E_G[X,Y]$ the set of
edges of $G$ with one end in $X$ and the other end in $Y$. If
$X=\{x\}$, we simply write $E_G[x,Y]$ for $E_G[\{x\},Y]$.

The generalized connectivity of a graph $G$, introduced by Chartrand
et al. in \cite{Chartrand1}, is a natural and nice generalization of
the concept of (vertex-)connectivity. For a graph $G=(V,E)$ and a
set $S\subseteq V(G)$ of at least two vertices, \emph{an $S$-Steiner
tree} or \emph{a Steiner tree connecting $S$} (or simply, \emph{an
$S$-tree}) is a such subgraph $T=(V',E')$ of $G$ that is a tree with
$S\subseteq V'$. Two Steiner trees $T$ and $T'$ connecting $S$ are
said to be \emph{internally disjoint} if $E(T)\cap
E(T')=\varnothing$ and $V(T)\cap V(T')=S$. For $S\subseteq V(G)$ and
$|S|\geq 2$, the \emph{generalized local connectivity} $\kappa(S)$
is the maximum number of internally disjoint trees connecting $S$ in
$G$. Note that when $|S|=2$ a Steiner tree connecting $S$ is just a
path connecting the two vertices of $S$. For an integer $k$ with
$2\leq k\leq n$, the \emph{generalized $k$-connectivity}
$\kappa_k(G)$ of $G$ is defined as $\kappa_k(G)= min\{\kappa(S) :
S\subseteq V(G) \ and \ |S|=k\}$. Clearly, when $|S|=2$,
$\kappa_2(G)$ is nothing new but the connectivity $\kappa(G)$ of
$G$, that is, $\kappa_2(G)=\kappa(G)$, which is the reason why one
addresses $\kappa_k(G)$ as the generalized connectivity of $G$. By
convention, for a connected graph $G$ with less than $k$ vertices,
we set $\kappa_k(G)=1$. Set $\kappa_k(G)=0$ when $G$ is
disconnected. Results on the generalized connectivity can be found
in \cite{Chartrand2, LLSun, LLL1, LLL2, LL, LLZ, Okamoto}.

As a natural counterpart of the generalized connectivity, we
introduced the concept of generalized edge-connectivity in
\cite{LMS}. For $S\subseteq V(G)$ and $|S|\geq 2$, the
\emph{generalized local edge-connectivity} $\lambda(S)$ is the
maximum number of edge-disjoint trees connecting $S$ in $G$. For an
integer $k$ with $2\leq k\leq n$, the \emph{generalized
$k$-edge-connectivity} $\lambda_k(G)$ of $G$ is then defined as
$\lambda_k(G)= min\{\lambda(S) : S\subseteq V(G) \ and \ |S|=k\}$.
It is also clear that when $|S|=2$, $\lambda_2(G)$ is nothing new
but the standard edge-connectivity $\lambda(G)$ of $G$, that is,
$\lambda_2(G)=\lambda(G)$, which is the reason why we address
$\lambda_k(G)$ as the generalized edge-connectivity of $G$. Also set
$\lambda_k(G)=0$ when $G$ is disconnected.

In addition to being natural combinatorial measures, the generalized
connectivity and generalized edge-connectivity can be motivated by
their interesting interpretation in practice. For example, suppose
that $G$ represents a network. If one considers to connect a pair of
vertices of $G$, then a path is used to connect them. However, if
one wants to connect a set $S$ of vertices of $G$ with $|S|\geq 3$,
then a tree has to be used to connect them. This kind of tree with
minimum order for connecting a set of vertices is usually called a
Steiner tree, and popularly used in the physical design of VLSI, see
\cite{Sherwani}. Usually, one wants to consider how tough a network
can be, for connecting a set of vertices. Then, the number of
totally independent ways to connect them is a measure for this
purpose. The generalized $k$-connectivity and generalized
$k$-edge-connectivity can serve for measuring the capability of a
network $G$ to connect any $k$ vertices in $G$.

The following two observations are easily seen.
\begin{observation}\label{obs1}
If $G$ is a connected graph, then $\kappa_k(G)\leq \lambda_k(G)\leq
\delta(G)$.
\end{observation}
\begin{observation}\label{obs2}
If $H$ is a spanning subgraph of $G$, then $\kappa_k(H)\leq
\kappa_k(G)$ and $\lambda_k(H)\leq \lambda_k(G)$.
\end{observation}

In \cite{LMS}, we obtained some results on the generalized
edge-connectivity. The following results are restated, which will be
used later.

\begin{lemma}\cite{LMS}\label{lem1}
For every two integers $n$ and $k$ with $2\leq k\leq n$,
$\lambda_k(K_n)=n-\lceil k/2\rceil.$
\end{lemma}

\begin{lemma}\cite{LMS}\label{lem2}
For any connected graph $G$, $\lambda_k(G)\leq \lambda(G)$.
Moreover, the upper bound is sharp.
\end{lemma}

\begin{lemma}\cite{LMS}\label{lem3}
Let $k,n$ be two integers with $2\leq k\leq n$. For a connected
graph $G$ of order $n$, $1\leq \lambda_k(G)\leq n-\lceil k/2
\rceil$. Moreover, the upper and lower bounds are sharp.
\end{lemma}

In \cite{LMS}, we characterized graphs with large generalized
$3$-connectivity and obtained the following result.

\begin{lemma}\cite{LMS}\label{lem4}
Let $k,n$ be two integers with $2\leq k\leq n$. For a connected
graph $G$ of order $n$, $\lambda_k(G)=n-\lceil\frac{k}{2}\rceil$ if
and only if $G=K_n$ for $k$ even; $G=K_n\setminus M$ for $k$ odd,
where $M$ is an edge set such that $0\leq |M|\leq \frac{k-1}{2}$.
\end{lemma}

Like \cite {LLMS}, here we will consider the generalized
$3$-edge-connectivity. From Lemma \ref{lem3}, $1\leq
\lambda_3(G)\leq n-2$. In Section $3$, graphs of order $n$ such that
$\lambda_3(G)= n-3$ is characterized.

Let $g(n,k,\ell)$ be the minimal number of edges of a graph $G$ of
order $n$ with $\lambda_k(G)=\ell \ (1\leq \ell\leq
n-\lceil\frac{k}{2}\rceil)$. From Lemma \ref{lem4}, we know that
$g(n,k,n-\lceil\frac{k}{2}\rceil)={n\choose 2}$ for $k$ even;
$g(n,k,n-\lceil\frac{k}{2}\rceil)={n\choose 2}-\frac{k-1}{2}$ for
$k$ odd. It is not easy to determine exact value of the parameter
$g(n,k,\ell)$. So we put our attention to on the case $k=3$. The
exact value of $g(n,3,\ell)$ for $\ell=n-2,n-3,1$ are obtained in
Section $4$. We also give a sharp lower bounds of $g(n,3,\ell)$ for
general $2\leq \ell\leq n-4$.

\section{Graphs with $\lambda_3(G)=n-3$}

After the preparation of the above section, we start to give our
main result. From Lemma \ref{lem3}, we know that for a connected
graph of order $G$ $1\leq \lambda_k(G)\leq
n-\lceil\frac{k}{2}\rceil$. Graphs with
$\lambda_k(G)=n-\lceil\frac{k}{2}\rceil$ has been shown in Lemma
\ref{lem4}. But, it is not easy to characterize graphs with
$\lambda_k(G)=n-\lceil\frac{k}{2}\rceil-1$ for general $k$. So we
focus on the case that $k=3$ and characterizing the graphs with
$\lambda_3(G)=n-3$ in this section.

For the generalized $3$-connectivity, we got the following result in
\cite{LLMS}.

\begin{theorem}\cite{LLMS}\label{th5}
Let $G$ be a connected graph of order $n \ (n\geq 3)$.
$\kappa_3(G)=n-3$ if and only if $\overline{G}=P_4\cup (n-4)K_1$ or
$\overline{G}=P_3\cup iP_2\cup (n-2i-3)K_1 \ (i=0, 1)$ or
$\overline{G}=C_3\cup iP_2\cup (n-2i-3)K_1 \ (i=0, 1)$ or
$\overline{G}=r P_2\cup (n-2r)K_1 \ ( 2\leq r\leq
\lfloor\frac{n}{2}\rfloor)$.
\end{theorem}

But, for the edge case we will show that the statement is different.
Before giving our main result, we need some preparations.

Choose $S\subseteq V(G)$. Then let $\mathscr{T}$ be a maximum set of
edge-disjoint trees connecting $S$ in $G$. Let $\mathscr{T}_1$ be
the set of trees in $\mathscr{T}$ whose edges belong to $E(G[S])$,
and let $\mathscr{T}_2$ be the set of trees containing at least one
edge of $E[S,\bar{S}]$. Thus $\mathscr{T}=\mathscr{T}_1\cup
\mathscr{T}_2$.

In \cite{LMS}, we obtained the following useful lemma.

\begin{lemma}\cite{LMS}\label{lem5}
Let $S\subseteq V(G)$, $|S|=k$ and $T$ be a tree connecting $S$. If
$T\in \mathscr{T}_1$, then $T$ uses $k-1$ edges of $E(G[S])\cup
E_G[S,\bar{S}]$; If $T\in \mathscr{T}_2$, then $T$ uses at least $k$
edges of $E(G[S])\cup E_G[S,\bar{S}]$.
\end{lemma}

By Lemma \ref{lem6}, we can derived the following result.

\begin{lemma}\label{lem6}
Let $G$ be a connected graph of order $n \ (n\geq 3)$, and $\ell$ be
a positive integer. If we can find a set $S\subseteq V(G)$ with
$|S|=3$ satisfying one of the following conditions, then
$\lambda_3(G)\leq n-\ell$.

$(1)$ $\overline{G}[S]=3K_1$ and $|E_{\overline{G}}[S,\bar{S}]\cup
\overline{G}[S]|\geq 3\ell -7$;

$(2)$ $\overline{G}[S]=P_2\cup K_1$ and
$|E_{\overline{G}}[S,\bar{S}]\cup \overline{G}[S]|\geq 3\ell -8$;

$(3)$ $\overline{G}[S]=P_3$ and $|E_{\overline{G}}[S,\bar{S}]\cup
\overline{G}[S]|\geq 3\ell -10$;

$(4)$ $\overline{G}[S]=K_3$ and $|E_{\overline{G}}[S,\bar{S}]\cup
\overline{G}[S]|\geq 3\ell -11$.
\end{lemma}
\begin{proof}
We only show that $(1)$ and $(3)$ hold, $(2)$ and $(4)$ can be
proved similarly.

$(1)$ Since $|E_{\overline{G}}[S,\bar{S}]\cup \overline{G}[S]|\geq
3\ell -7$, we have $|E(G[S])\cup E_G[S,\bar{S}]|\leq 3+3(n-3)-(3\ell
-7)=3n-3\ell +1$. Since $\overline{G}[S]=3K_1$, we have $G[S]=K_3$.
Therefore, $|E(G[S])|=3$, and so there exists at most one tree
belonging to $\mathscr{T}_1$ in $G$. If there exists one tree
belonging to $\mathscr{T}_1$, namely $|\mathscr{T}_1|=1$, then the
other trees connecting $S$ must belong to $\mathscr{T}_2$. From
Lemma \ref{lem6}, each tree belonging to $\mathscr{T}_2$ uses at
least $3$ edges in $E(G[S])\cup E_G[S,\bar{S}]$. So the remaining at
most $(3n-3\ell +1)-2$ edges of $E(G[S])\cup E_G[S,\bar{S}]$ can
form at most $\frac{3n-3\ell -1}{3}$ trees. Thus $\lambda_3(G)\leq
\lambda(S)=
|\mathscr{T}|=|\mathscr{T}_1|+|\mathscr{T}_2|=1+|\mathscr{T}_2|\leq
n-\ell+\frac{2}{3}$, which results in $\lambda_3(G)\leq n-\ell$
since $\lambda_3(G)$ is an integer. Suppose that all trees
connecting $S$ belong to $\mathscr{T}_2$. Then $\lambda(S)=
|\mathscr{T}|=|\mathscr{T}_2|\leq \frac{3n-3\ell +1}{3}$, which
implies that $\lambda_3(G)\leq \lambda(S)=n-\ell$.

$(3)$ Since $|E_{\overline{G}}[S,\bar{S}]\cup \overline{G}[S]|\geq
3\ell -10$, we have $|E(G[S])\cup E_G[S,\bar{S}]|\leq
3+3(n-3)-(3\ell -10)=3n-3\ell +2$. Since $\overline{G}[S]=P_3$, we
have $G[S]=P_2\cup K_1$. Since $|E(G[S])|=1$, there exists no tree
belonging to $\mathscr{T}_1$. So each tree connecting $S$ must
belong to $\mathscr{T}_2$. From Lemma \ref{lem6}, $\lambda(S)\leq
|\mathscr{T}|=|\mathscr{T}_2|\leq \frac{3n-3\ell +2}{3}$, which
implies that $\lambda_3(G)\leq \lambda(S)=n-\ell$ since
$\lambda_3(G)$ is an integer.
\end{proof}

\begin{lemma}\label{lem7}
Let $G$ be a connected graph with minimum degree $\delta$. If there
are two adjacent vertices of degree $\delta$, then $\lambda_k(G)\leq
\delta(G)-1$.
\end{lemma}
\begin{proof}
It is clear that $\lambda(G)\leq \delta$ and $\lambda_k(G)\leq
\lambda(G)$ by Lemma \ref{lem2}. So $\lambda_k(G)\leq \delta$.

Suppose that there are two adjacent vertices $v_1$ and $v_2$ of
degree $\delta$ and $\delta$. Besides $v_1$ and $v_2$, we choose a
vertex $v_3$ in $V(G\setminus \{v_1,v_2\})$ to get a $k$-set $S$
containing $v_1,v_2,v_3$. Suppose $T_1,T_2,\cdots,T_{\delta}$ are
$\delta$ pairwise edge-disjoint trees connecting $S$. Since $G$ is
simple graph, obviously the $\delta$ edges incident $v_1$ must be
contained in $T_1,T_2,\cdots,T_{\delta}$, respectively, and so are
the $\delta$ edges incident $v_2$. Without loss of generality, we
may assume that the edge $v_1v_2$ is contained in $T_1$. But, since
$T_1$ is a tree connecting $S$, it must contain another edge
incident with $v_1$ or $v_2$, a contradiction. Thus
$\lambda_k(G)\leq \delta-1$.
\end{proof}

A subset $M$ of $E(G)$ is called a \emph{matching} of $G$ if the
edges of $M$ satisfy that no two of them are adjacent in $G$. A
matching $M$ saturates a vertex $v$, or $v$ is said to be
\emph{$M$-saturated}, if some edge of $M$ is incident with $v$;
otherwise, $v$ is \emph{$M$-unsaturated}. $M$ is a \emph{maximum
matching} if $G$ has no matching $M'$ with $|M'|>|M|$.

\begin{theorem}\label{th2}
Let $G$ be a connected graph of order $n$. Then $\lambda_3(G)=n-3$
if and only if $\overline{G}=r P_2\cup (n-2r)K_1 \ (2\leq r\leq
\lfloor\frac{n}{2}\rfloor)$ or $\overline{G}=P_4\cup sP_2\cup
(n-2s-4)K_1 \ (0\leq s\leq \lfloor\frac{n-4}{2}\rfloor)$ or
$\overline{G}=P_3\cup t P_2\cup (n-2t-3)K_1 \ (0\leq t\leq
\lfloor\frac{n-3}{2}\rfloor)$ or $\overline{G}=C_3\cup t P_2\cup
(n-2t-3)K_1 \ (0\leq t\leq \lfloor\frac{n-3}{2}\rfloor)$.
\end{theorem}
\begin{proof}
\emph{Sufficiency:} Assume that $\lambda_3(G)=n-3$. From Lemma
\ref{lem4}, for a connected graph $H$, $\lambda_3(H)=n-2$ if and
only if $|E(\overline{H})|=1$. Since $\lambda_3(G)=n-3$, it follows
that $|E(\overline{G})|\geq 2$. We claim that
$\delta(\overline{G})\leq 2$. Assume, to the contrary, that
$\delta(\overline{G})\geq 3$. Then $\lambda_3(G)\leq
\delta(G)=n-1-\delta(\overline{G})\leq n-4$, a contradiction. Since
$\delta(\overline{G})\leq 2$, it follows that each component of
$\overline{G}$ is a path or a cycle (note that a isolated vertex in
$\overline{G}$ is a trivial path). We will show that the following
claims hold.

\textbf{ Claim 1.}~~$\overline{G}$ has at most one component of
order larger than $2$.

Suppose, to the contrary, that $\overline{G}$ has two components of
order larger than $2$, denoted by $H_1$ and $H_2$ (see Figure 1
$(a)$).

Let $x,y\in V(H_1)$ and $z\in V(H_2)$ such that
$d_{H_1}(y)=d_{H_2}(z)=2$ and $x$ is adjacent to $y$ in $H_1$. Thus
$d_{G}(y)=n-1-d_{\overline{G}}(y)=n-1-d_{H_1}(y)=n-3$. The same is
true for $z$, that is, $d_{G}(z)=n-3$. Pick $S=\{x,y,z\}$. This
implies that $\delta(G)\leq d_{G}(z)\leq n-3$. Since all other
components of $\overline{G}$ are paths or cycles, $\delta(G)\geq
n-3$. So $\delta(G)=n-3$ and hence
$d_{G}(y)=d_{G}(z)=\delta(G)=n-3$. Since $yz\in E(G)$, by Lemma
\ref{lem7} it follows that $\lambda_3(G)\leq \delta(G)-1=n-4$, a
contradiction.

\textbf{Claim 2.}~~If $H$ is the component of $\overline{G}$ of
order larger than $3$, then $H$ is a $4$-path.

Assume, to the contrary, that $H$ is a path or a cycle of order
larger than $4$, or a cycle of order $4$.

First, we consider the former. We can pick a $P_5$ in $H$. Let
$P_5=v_1,v_2,v_3,v_4,v_5$, $S=\{v_2,v_3,v_4\}$ and
$\bar{S}=G\setminus \{v_2,v_3,v_4\}$ (see Figure 1 $(b)$). Since
$v_2v_3,v_3v_4\notin E(G[S])$, there exists no tree of type $I$
connecting $S$. From Lemma \ref{lem5}, each tree of type $II$ uses
at least $3$ edges. Since $|E(G[S])\cup E_G[S,\bar{S}]|=3(n-3)-1$,
we have $|\mathscr{T}_2|\leq \frac{3(n-3)-1}{3}$ and hence
$|\mathscr{T}|=|\mathscr{T}_2|=n-4$ since $\lambda_3(G)$ is an
integer. This contradicts to $\lambda_3(G)=n-3$.

\begin{figure}[h,t,b,p]
\begin{center}
\scalebox{0.8}[0.8]{\includegraphics{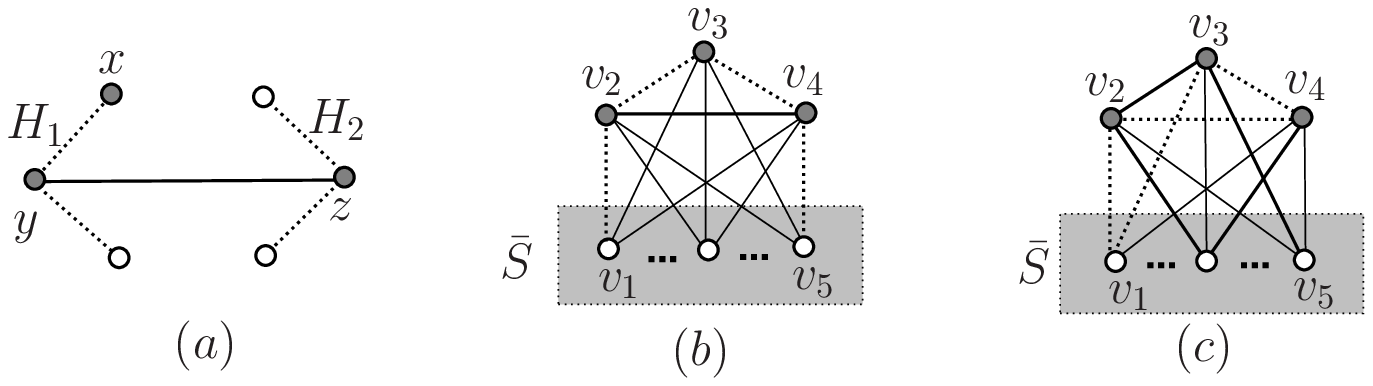}}\\
Figure 1: Graphs for Claims $1$ and $2$.
\end{center}
\end{figure}

Now we consider the latter. Let $H=v_1,v_2,v_3,v_4$ be a cycle, and
$S=\{v_2,v_3,v_4\}$ (see Figure 1 $(c)$). Since $v_2v_3,v_3v_4\notin
E[S]$, there exists no tree of type $I$. Since each tree of type
$II$ uses at least $3$ edges and $|E(G[S])\cup
E_G[S,\bar{S}]|=3(n-3)-1$, we have $|\mathscr{T}_2|\leq
\frac{3(n-3)-1}{3}$ and $|\mathscr{T}|=|\mathscr{T}_2|=n-4$, which
also contradicts to $\lambda_3(G)=n-3$.

From the above two claims, we know that if $\overline{G}$ has a
component $P_4$, then it is the only component of order larger than
$3$ and the other components must be independent edges. Let $s$ be
the number of such independent edges. $\overline{G}$ can have as
many as such independent edges, which implies that $s\leq
\lfloor\frac{n-4}{2}\rfloor$. From Lemma \ref{lem4}, $s\geq 0$. Thus
$0\leq s\leq \lfloor\frac{n-4}{2}\rfloor$.

By the similar analysis, we conclude that $\overline{G}=r P_2\cup
(n-2r)K_1 \ (2\leq r\leq \lfloor\frac{n}{2}\rfloor)$ or
$\overline{G}=P_4\cup s P_2\cup (n-2s-4)K_1 \ (0\leq s\leq
\lfloor\frac{n-4}{2}\rfloor)$ or $\overline{G}=P_3 \cup t P_2\cup
(n-2t-3)K_1 \ (0\leq t\leq \lfloor\frac{n-3}{2}\rfloor)$ or
$\overline{G}=C_3 \cup t P_2\cup (n-2t-3)K_1 \ (0\leq t\leq
\lfloor\frac{n-3}{2}\rfloor)$.

\emph{Necessity:} We will show that $\lambda_3(G)\geq n-3$ if $G$ is
a graph with the conditions of this theorem. We have the following
cases to consider.

\textbf{Case 1.}~~$\overline{G}=P_3 \cup t P_2\cup (n-2t-3)K_1$ or
$\overline{G}=C_3 \cup t P_2\cup (n-2t-3)K_1 \ (0\leq t\leq
\lfloor\frac{n-3}{2}\rfloor)$.

We only need to show that $\lambda_3(G)\geq n-3$ for
$t=\lfloor\frac{n-3}{2}\rfloor$. If $\lambda_3(G)\geq n-3$ for
$\overline{G}=C_3 \cup t P_2\cup (n-2t-3)K_1$, then
$\lambda_3(G)\geq n-3$ for $\overline{G}=P_3 \cup t P_2\cup
(n-2t-3)K_1$. It suffices to check that $\lambda_3(G)\geq n-3$ for
$\overline{G}=C_3 \cup \lfloor\frac{n-3}{2}\rfloor P_2\cup
(n-2\lfloor\frac{n-3}{2}\rfloor-3)K_1$.

Let $C_3=v_1,v_2,v_3$ and $S=\{x,y,z\}$ be a $3$-subset of $G$, and
$M=\lfloor\frac{n-3}{2}\rfloor P_2$. It is clear that $M$ is a
maximum matching of $\overline{G}\setminus V(C_3)$. Then
$\overline{G}\setminus V(C_3)$ has at most one $M$-unsaturated
vertex.

\begin{figure}[h,t,b,p]
\begin{center}
\scalebox{0.8}[0.8]{\includegraphics{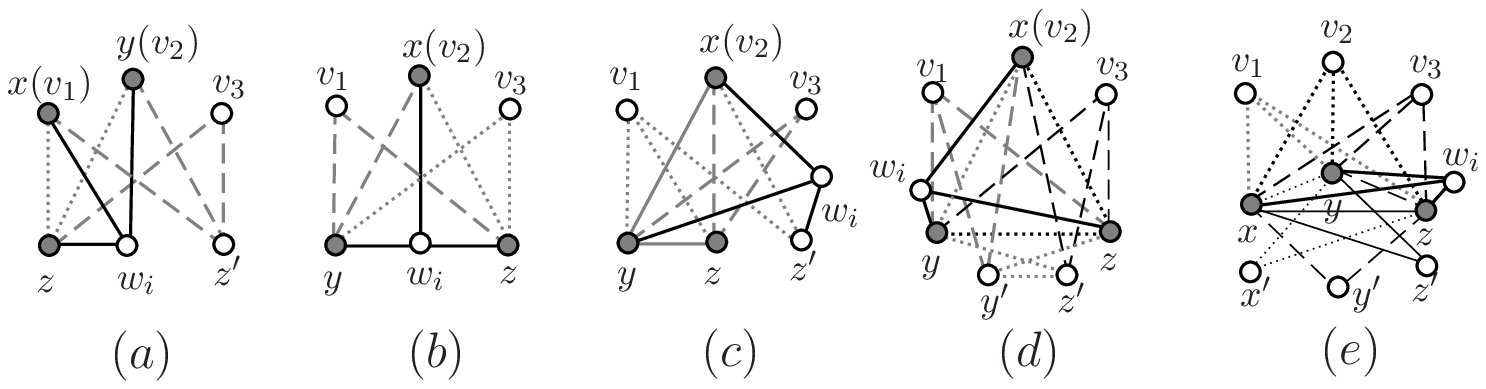}}\\
Figure 2: Graphs for Case $1$.
\end{center}
\end{figure}

If $S=V(C_3)$, then there exist $n-3$ pairwise edge-disjoint trees
connecting $S$ since each vertex in $S$ is adjacent to every vertex
in $G\setminus S$. Suppose $S\neq V(C_3)$.

If $|S\cap V(C_3)|=2$, then one element of $S$ belongs to $\in
V(G)\setminus V(C_3)$, denoted by $z$. Since
$d_G(v_1)=d_G(v_2)=d_G(v_3)=n-3$, we can assume that $x=v_1$,
$y=v_2$. When $z$ is $M$-unsaturated, the trees $T_i=w_ix\cup
w_iy\cup w_iz$ together with $T_1=xz\cup yz$ form $n-3$ pairwise
edge-disjoint trees connecting $S$, where
$\{w_1,w_2,\cdots,w_{n-4}\}=V(G)\setminus \{x,y,z,v_3\}$. When $z$
is $M_1$-saturated, we let $z'$ be the adjacent vertex of $z$ under
$M_1$. Then the trees $T_i=w_ix\cup w_iy\cup w_iz$ together with
$T_1=xz\cup yz$ and $T_2=xz'\cup yz'\cup z'v_3\cup zv_3$ form $n-3$
pairwise edge-disjoint trees connecting $S$ (see Figure 2 $(a)$),
 where $\{w_1,w_2,\cdots,w_{n-5}\}=V(G)\setminus \{x,y,z,z',v_3\}$.

If $|S\cap V(C_3)|=1$, then two elements of $S$ belong to $\in
V(G)\setminus V(C_3)$, denoted by $y$ and $z$. Without loss of
generality, let $x=v_2$. When $y$ and $z$ are adjacent under $M_1$,
the trees $T_i=w_ix\cup w_iy\cup w_iz$ together with $T_1=xy\cup
yv_1\cup v_1z$ and $T_2=xz\cup zv_3\cup v_3y$ form $n-3$ pairwise
edge-disjoint trees connecting $S$ (see Figure 2 $(b)$), where
$\{w_1,w_2,\cdots,w_{n-5}\}=V(G)\setminus \{x,y,z,v_1,v_3\}$. When
$y$ and $z$ are nonadjacent under $M$, we consider whether $y$ and
$z$ are $M$-saturated. If one of $\{y,z\}$ is $M$-unsaturated,
without loss of generality, we assume that $y$ is $M$-unsaturated.
Since $G\setminus V(C_3)$ has at most one $M$-unsaturated vertex,
$z$ is $M$-saturated. Let $z'$ be the adjacent vertex of $z$ under
$M$. Then the trees $T_i=w_ix\cup w_iy\cup w_iz$ together with
$T_1=xy\cup yz$ and $T_2=v_1y\cup v_1z\cup z'v_1\cup z'x$ and
$T_3=xz\cup zv_3\cup v_3y$ form $n-3$ pairwise edge-disjoint trees
connecting $S$ (see Figure 2 $(c)$), where
$\{w_1,w_2,\cdots,w_{n-6}\}=V(G)\setminus \{x,y,z,z',v_1,v_3\}$. If
both $y$ and $z$ are $M$-saturated, we let $y',z'$ be the adjacent
vertex of $y,z$ under $M$, respectively. Then the trees
$T_i=w_ix\cup w_iy\cup w_iz$ together with $T_1=xz\cup yz$,
$T_2=xy\cup yz'\cup z'y'\cup y'z$, $T_3=yv_3\cup z'v_3\cup zv_3\cup
xz'$ and $T_4=yv_1\cup y'v_1\cup zv_1\cup y'x$ form $n-3$ pairwise
edge-disjoint trees connecting $S$ (see Figure 2 $(d)$), where
$\{w_1,w_2,\cdots,w_{n-7}\}=V(G)\setminus \{x,y,z,y',z',v_1,v_3\}$.

Otherwise, $S\subseteq G\setminus V(C_3)$. When one of $\{x,y,z\}$
is $M$-unsaturated, without loss of generality, we assume that $x$
is $M$-unsaturated. Since $G\setminus V(C_3)$ has at most one
$M$-unsaturated vertex, both $y$ and $z$ are $M$-saturated. Let
$y',z'$ be the adjacent vertex of $y,z$ under $M$, respectively. We
pick a vertex $x'$ of $V(G)\setminus \{x,y,y',z,z',v_1,v_2,v_3\}$.
When $x,y,z$ are all $M$-saturated, we let $x',y',z'$ be the
adjacent vertex of $x,y,z$ under $M$, respectively. Then the trees
$T_i=w_ix\cup w_iy\cup w_iz$ together with $T_j=xv_j\cup yv_j\cup
zv_j(1\leq j\leq 3)$ and $T_4=xy\cup yx'\cup x'z$ and $T_5=xy'\cup
zy'\cup zy$ and $T_6=zx\cup xz'\cup z'y$ form $n-3$ pairwise
edge-disjoint trees connecting $S$ (see Figure 2 $(e)$), where
$\{w_1,w_2,\cdots,w_{n-9}\}=V(G)\setminus
\{x,y,z,x',y',z',v_1,v_2,v_3\}$.

From the above discussion, we get that $\lambda(S)\geq n-3$ for
$S\subseteq V(G)$, which implies $\lambda_3(G)\geq n-3$. So
$\lambda_3(G)=n-3$.

\textbf{\itshape Case 2.}~~$\overline{G}=r P_2\cup (n-2r)K_1 \
(2\leq r\leq \lfloor\frac{n}{2}\rfloor)$ or $\overline{G}=P_4\cup s
P_2\cup (n-2s-4)K_1 \ (0\leq s\leq \lfloor\frac{n-4}{2}\rfloor)$.

We only need to show that $\lambda_3(G)\geq n-3$ for
$r=\lfloor\frac{n}{2}\rfloor$ and $s=\lfloor\frac{n-4}{2}\rfloor$.
If $\lambda_3(G)\geq n-3$ for $\overline{G}=P_4\cup
\lfloor\frac{n-4}{2}\rfloor P_2\cup
(n-2\lfloor\frac{n-4}{2}\rfloor-4) K_1$, then $\lambda_3(G)\geq n-3$
for $\overline{G}=\lfloor\frac{n}{2}\rfloor P_2\cup
(n-2\lfloor\frac{n}{2}\rfloor)K_1$. So we only need to consider the
former. Let $P_4=v_1,v_2,v_3,v_4$, $S=\{x,y,z\}$ be a $3$-subset of
$G$, and $M=\overline{G}\setminus E(P_4)$. Clearly, $M$ is a maximum
matching of $\overline{G}\setminus V(P_4)$. It is easy to see that
$\overline{G}\setminus V(P_4)$ has at most one $M$-unsaturated
vertex. For any $S\subseteq V(G)$, we will show that there exist
$n-3$ edge-disjoint trees connecting $S$ in $G$.

If $S\subseteq V(P_4)$, then there exist $n-4$ pairwise
edge-disjoint trees connecting $S$ since each vertex in $S$ is
adjacent to every vertex in $G\setminus V(P_4)$. Since
$d_G(v_1)=d_G(v_4)=n-2$ and $d_G(v_2)=d_G(v_3)=n-3$, we only need to
consider $S=\{v_1,v_2,v_3\}$ and $S=\{v_1,v_2,v_4\}$. These trees
together with $T=yv_4\cup v_4x\cup v_4z$ for $S=\{v_1,v_2,v_3\}$, or
$T=xy\cup yz$ for $S=\{v_1,v_2,v_3\}$ form $n-3$ pairwise
edge-disjoint trees connecting $S$. Suppose $S\cap V(P_4)\neq 3$.

\begin{figure}[h,t,b,p]
\begin{center}
\scalebox{0.8}[0.8]{\includegraphics{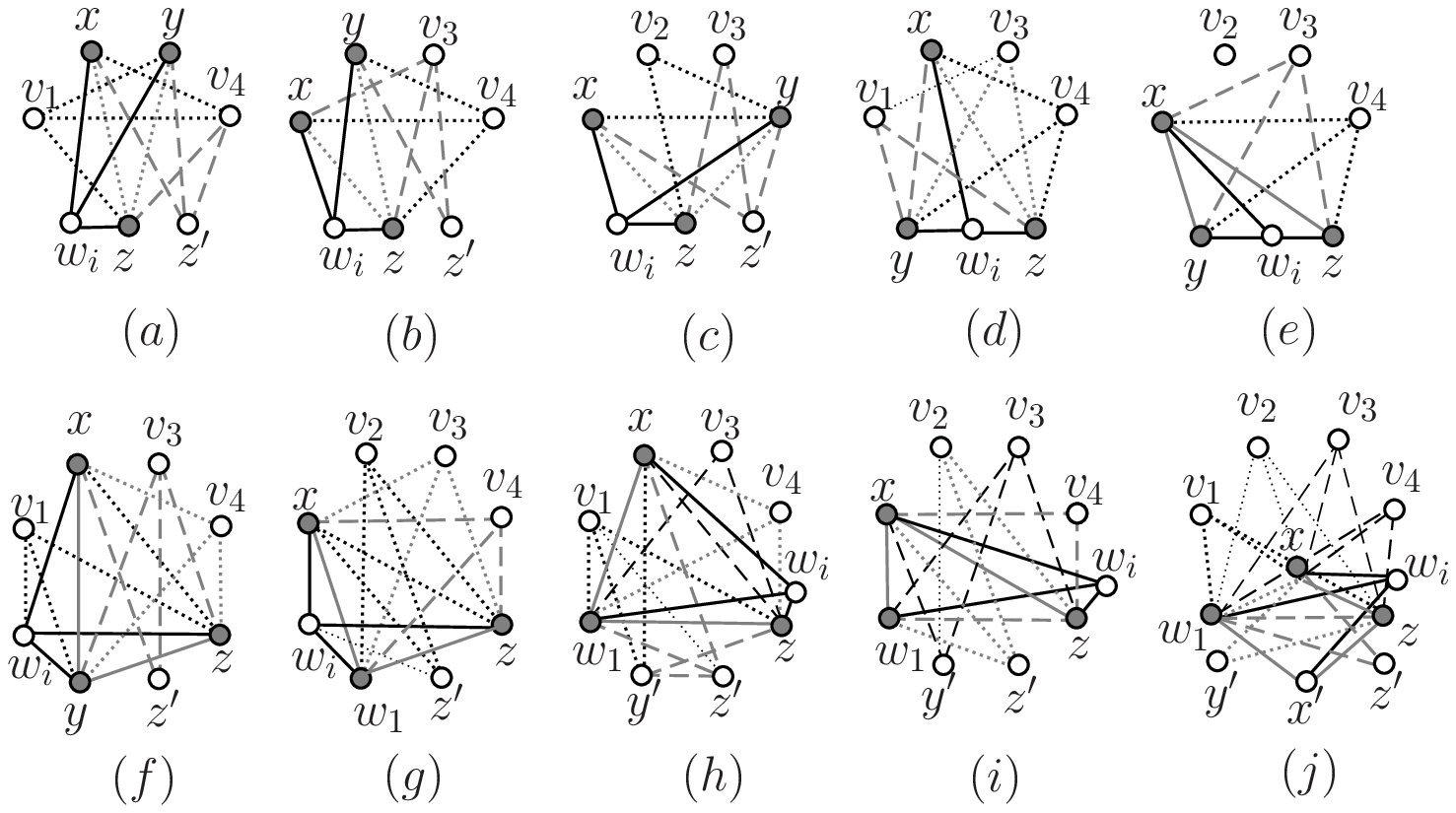}}\\
Figure 3: Graphs for $S$ in Case 2.
\end{center}
\end{figure}

If $|S\cap V(P_4)|=2$, then one element of $S$ belongs to $\in
V(G)\setminus V(P_4)$, denoted by $z$. Since $d_G(v_1)=d_G(v_4)=n-2$
and $d_G(v_2)=d_G(v_3)=n-3$, we only need to consider $x=v_1,y=v_2$
or $x=v_2,y=v_3$ or $x=v_1,y=v_4$. When $z$ is $M$-unsaturated, the
trees $T_i=w_ix\cup w_iy\cup w_iz$ together with $T_1=xz\cup yz$,
$T_2=xv_4\cup yv_4\cup zv_4$ for $x=v_1,y=v_2$, or $T_2=xv_4\cup
v_4v_1\cup v_1y\cup v_4z$ for $x=v_2,y=v_3$, or $T_2=xv_3\cup
yv_3\cup zv_3$ for $x=v_1,y=v_4$ form $n-3$ pairwise edge-disjoint
trees connecting $S$, where
$\{w_1,w_2,\cdots,w_{n-5}\}=V(G)\setminus (V(P_4)\cup\{z\})$. When
$z$ is $M$-unsaturated, we let $z'$ be the adjacent vertex of $z$
under $M$. For $x=v_2,y=v_3$, the trees $T_i=w_ix\cup w_iy\cup w_iz$
together with $T_1=xz\cup yz$, $T_2=xz'\cup yz'\cup z'v_4\cup zv_4$
and $T_2=yv_1\cup v_1v_4\cup zv_1\cup xv_4$ form $n-3$ pairwise
edge-disjoint trees connecting $S$ (see Figure 3 $(a)$), where
$\{w_1,w_2,\cdots,w_{n-6}\}=V(G)\setminus \{x,y,z,z',v_1,v_4\}$. One
can check that the same is true for $x=v_1,y=v_2$ and $x=v_1,y=v_4$
(see Figure 3 $(b)$ and $(c)$).

If $|S\cap V(P_4)|=1$, then two elements of $S$ belong to $\in
V(G)\setminus V(P_4)$, denoted by $y$ and $z$. We only need to
consider $x=v_1$ or $x=v_2$. When $y$ and $z$ are adjacent under
$M_1$, the trees $T_i=w_ix\cup w_iy\cup w_iz$ together with
$T_1=xy\cup zv_1\cup yv_1$, $T_2=xz\cup zv_3\cup yv_3$ and
$T_3=xv_4\cup yv_4\cup zv_4$ form $n-3$ pairwise edge-disjoint trees
connecting $S$ for $x=v_2$ (see Figure 3 $(d)$), where
$\{w_1,w_2,\cdots,w_{n-6}\}=V(G)\setminus \{x,y,z,v_1,v_3,v_4\}$.
The same is true for $x=v_1$ (see Figure 3 $(e)$). When $y$ and $z$
are nonadjacent under $M$, we consider whether $y$ and $z$ are
$M$-saturated. If one of $\{y,z\}$ is $M$-unsaturated, without loss
of generality, we assume that $y$ is $M$-unsaturated. Since
$G\setminus V(P_4)$ has at most one $M$-unsaturated vertex, $z$ is
$M$-saturated. Let $z'$ be the adjacent vertex of $z$ under $M$. For
$x=v_2$, the trees $T_i=w_ix\cup w_iy\cup w_iz$ together with
$T_1=xz\cup yz$, $T_2=v_4x\cup v_4y\cup v_4z$, $T_3=v_1y\cup
v_1z\cup zx$ and $T_4=z'x\cup v_3y\cup z'v_3\cup zv_3$ form $n-3$
pairwise edge-disjoint trees connecting $S$ (see Figure 3 $(f)$),
where $\{w_1,w_2,\cdots,w_{n-7}\}=V(G)\setminus
\{x,y,z,z',v_1,v_3,v_4\}$. The same is true for $x=v_1$ (see Figure
3 $(g)$). If both $y$ and $z$ are $M$-saturated, we let $y',z'$ be
the adjacent vertex of $y,z$ under $M$, respectively. For $x=v_2$,
the trees $T_i=w_ix\cup w_iy\cup w_iz$ together with $T_1=xz\cup
yz$, $T_2=yv_3\cup zv_3\cup zx$, $T_3=xv_4\cup yv_4\cup zv_4$,
$T_4=yv_1\cup y'v_1\cup zv_1\cup xy'$ and $T_5=xz'\cup z'y\cup
z'y'\cup y'z$ form $n-3$ pairwise edge-disjoint trees connecting $S$
(see Figure 3 $(h)$), where
$\{w_1,w_2,\cdots,w_{n-8}\}=V(G)\setminus
\{x,y,z,y',z',v_1,v_3,v_4\}$. The same is true for $x=v_1$ (see
Figure 3 $(i)$).

If $S\subseteq G\setminus V(P_4)$, when one of $\{x,y,z\}$ is
$M$-unsaturated, without loss of generality, we let $x$ is
$M$-unsaturated, then both $y$ and $z$ are $M$-saturated. Let
$y',z'$ be the adjacent vertex of $y,z$ under $M$, respectively. We
pick a vertex $x'$ of $V(G)\setminus \{x,y,y',z,z',v_1,v_2,v_3\}$.
When $x,y,z$ are all $M$-saturated, we let $x',y',z'$ be the
adjacent vertex of $x,y,z$ under $M$, respectively. Then the trees
$T_i=w_ix\cup w_iy\cup w_iz$ together with $T_j=xv_j\cup yv_j\cup
zv_j(1\leq j\leq 4)$ and $T_5=yx\cup xy'\cup y'z$ and $T_6=yx'\cup
zx'\cup zx$ and $T_7=zy\cup yz'\cup z'x$ form $n-3$ pairwise
edge-disjoint trees connecting $S$ (see Figure 3 $(j)$), where
$\{w_1,w_2,\cdots,w_{n-10}\}=V(G)\setminus
\{x,y,z,x',y',z',v_1,v_2,v_3,v_4\}$.

From the above argument, we conclude that for any $S\subseteq V(G)$
$\lambda(S)\geq n-3$. From the arbitrariness of $S$, we have
$\lambda_3(G)\geq n-3$. The proof is now complete.
\end{proof}

\section{The minimal size of a graph with $\lambda_3=\ell$}

Recall that $g(n,k,\ell)$ is the minimal number of edges of a graph
$G$ of order $n$ with $\lambda_k(G)=\ell \ (1\leq \ell\leq
n-\lceil\frac{k}{2}\rceil)$. Let us focus on the case $k=3$ and
derive the following result.

\begin{theorem}\label{th2}
Let $n$ be an integer with $n\geq 3$. Then

$(1)$ $g(n,3,n-2)={n\choose{2}}-1$;

$(2)$ $g(n,3,n-3)={n\choose{2}}-\lfloor\frac{n+3}{2}\rfloor$;

$(3)$ $g(n,3,1)=n-1$;

$(4)$ $g(n,3,\ell)\geq \big\lceil \frac{\ell(\ell+1)}{2\ell+1}n
\big\rceil$ for $n\geq 11$ and $1\leq \ell\leq n-4$. Moreover, the
bound is sharp.
\end{theorem}
\begin{proof}
$(1)$ From Lemma \ref{lem4}, $\lambda_3(G)=n-2$ if and only if
$G=K_n$ or $G=K_n\setminus e$ where $e\in E(K_n)$. So
$g(n,3,n-2)={n\choose{2}}-1$.

$(2)$ From Theorem \ref{th2}, $\lambda_3(G)=n-3$ if and only if
$\overline{G}=r P_2\cup (n-2r)K_1 \ (2\leq r\leq
\lfloor\frac{n}{2}\rfloor)$ or $\overline{G}=P_4\cup sP_2\cup
(n-2s-4)K_1 \ (0\leq s\leq \lfloor\frac{n-4}{2}\rfloor)$ or
$\overline{G}=P_3\cup t P_2\cup (n-2t-3)K_1 \ (0\leq t\leq
\lfloor\frac{n-3}{2}\rfloor)$ or $\overline{G}=C_3\cup t P_2\cup
(n-2t-3)K_1 \ (0\leq t\leq \lfloor\frac{n-3}{2}\rfloor)$. If $n$ is
even, then $max\{e(\overline{G})\}=\frac{n+2}{2}$, which implies
that
$g(n,3,n-3)={n\choose{2}}-max\{e(\overline{G})\}={n\choose{2}}-\frac{n+2}{2}$.
If $n$ is odd, then $max\{e(\overline{G})\}=\frac{n+3}{2}$, which
implies that
$g(n,3,n-3)={n\choose{2}}-max\{e(\overline{G})\}={n\choose{2}}-\frac{n+3}{2}$.
So $g(n,3,n-3)={n\choose{2}}-\lfloor\frac{n+3}{2}\rfloor$.

$(3)$ It is clear that the tree $T_n$ is the graph such that
$\lambda_3(G)=1$ with the minimal number of edges. So
$g(n,3,1)=n-1$.

$(4)$ Since $\lambda_k=\ell$, by Lemma \ref{lem5}, we know that
$\delta(G)\geq \ell$ and any two vertices of degree $\ell$ are not
adjacent. Denote by $X$ the set of vertices of degree $\ell$. We
have that $X$ is a independent set. Put $Y=V(G)\setminus X$ and
obviously there are $2|X|$ edges joining $X$ to $Y$. Assume that
$m'$ is the number of edges joining two vertices belonging to $Y$.
It is clear that

$$
e=\ell|X|+m' \eqno (1)
$$

Since every vertex of $Y$ has degree at least $\ell+1$ in $G$, then
$\sum_{v\in Y}d(v)=\ell|X|+2m'\geq (\ell+1)|Y|=(\ell+1)(n-|X|)$,
namely,

$$
(2\ell+1)|X|+ 2m'\geq (\ell+1)n \eqno (2)
$$

Combining $(1)$ with $(2)$, we have
$\frac{2\ell+1}{\ell}e(G)=(2\ell+1)|X|+\frac{2\ell+1}{\ell}m'\geq
(2\ell+1)|X|+2m'\geq (\ell+1)n$ Therefore, $e(G)\geq
\frac{\ell(\ell+1)}{2\ell+1}n$. Since the number of edges is an
integer, it follows that $e(G)\geq \lceil
\frac{\ell(\ell+1)}{2\ell+1}n \rceil$.

To show that the upper bound is sharp, we consider the complete
bipartite graph $G=K_{\ell,\ell+1}$. Let
$U=\{u_1,u_2,\cdots,u_{\ell}\}$ and
$W=\{w_1,w_2,\cdots,w_{\ell+1}\}$ be the two parts of
$K_{\ell,\ell+1}$. Choose $S\subseteq V(G)$. We will show that there
are $\ell$ edge-disjoint trees connecting $S$.

If $|S\cap U|=3$, without loss of generality, let
$S=\{u_1,u_2,u_{3}\}$, then the trees $T_i=u_1v_i\cup u_2v_i\cup
u_3v_i \ (1\leq i\leq \ell+1)$ are $\ell+1$ edge-disjoint trees
connecting $S$.

If $|S\cap U|=2$, then $|S\cap W|=1$. Without loss of generality,
let $S=\{u_1,u_2,v_{1}\}$. Then the trees $T_i=u_1v_i\cup u_iv_i\cup
u_iv_1 \ (4\leq i\leq \ell+1)$ and $T_1=u_1v_1\cup u_1v_3\cup
u_2u_3$ and $T_2=u_2v_1\cup u_2v_2\cup u_1v_2$ are $\ell$
edge-disjoint trees connecting $S$.

If $|S\cap U|=1$, then $|S\cap W|=2$. Without loss of generality,
let $S=\{u_1,v_1,v_{2}\}$. Then the trees $T_i=u_1v_{i+1}\cup
u_iv_{i+1}\cup u_iv_1\cup u_iv_2 \ (2\leq i\leq \ell)$ and
$T_1=u_1v_1\cup u_1v_2$ are $\ell$ edge-disjoint trees connecting
$S$.

Suppose $|S\cap W|=3$. Without loss of generality, let
$S=\{w_1,w_2,w_{3}\}$, then the trees $T_i=w_1u_i\cup w_2u_i\cup
w_3u_i \ (1\leq i\leq \ell)$ are $\ell$ edge-disjoint trees
connecting $S$.

From the above argument, we conclude that, for any $S\subseteq
V(G)$, $\lambda(S)\geq \ell$. So $\lambda_3(G)\geq \ell$. On the
other hand, $\lambda_3(G)\leq \delta(G)=\ell$ and hence
$\lambda_3(G)=\ell$. Clearly, $|V(G)|=2\ell+1$,
$e(G)=\ell(\ell+1)=\lceil \frac{\ell(\ell+1)}{2\ell+1}n \rceil$.

So the lower bound is sharp for $k=3$ and $2\leq \ell\leq
n-2-\lceil\frac{k}{2}\rceil$.

\end{proof}

\end{document}